\documentclass[12pt]{amsart}
\usepackage{amsmath,amsthm,amssymb,latexsym}
\usepackage{epsfig}
\usepackage{graphics}

\newtheorem{theorem}{Theorem}
\newtheorem{remark}{Remark}
\newtheorem{lemma}{Lemma}
\newtheorem{proposition}{Proposition}
\newtheorem{corollary}{Corollary}
\newtheorem{definition}{Definition}

\begin{document}
\title[Cable of the torus knots]
{On the colored Jones polynomials\\
       of certain cable of the torus knots}
\author{Qihou Liu}
\email{qihou.liu@yale.edu}
       
\date{July 25,2008}
\maketitle

\begin{abstract}
In this paper, we study the asymptotic behavior of the colored Jones polynomials evaluated at roots of unity for a special class of knots. We show that certain limit is zero as predicted by the volume conjecture.

\end{abstract}

\section{Introduction}\label{intro}
In\cite{KT}, Kashaev obtained an explicit expansion of the colored Jones polynomials evaluated at roots of unity for the torus knots. As a consequence, it is shown that the volume conjecture holds for such knots, ---that is:
\[\lim_{N\to \infty} \frac{\log\left|\left(\frac{J_N^\sigma
(A,T)}{[N+1]}\right)_{A=\exp\left(\frac{\pi i}{2(N+1)}\right)}\right|}N=0\]

Using a formula from Morton\cite{MO}, one can show that the quantity inside $\log$ function is bounded from above by a polynomial of $N$. To complete the argument, however, one need to bound this quantity from below. In a recent preprint by Roland van der Veen~\cite{VV}, the author claims a proof for those knots with volume zero. An upper bound is obtained. We still need to see how the lower bound would be treated. In this paper, we extend Kashaev's method to give an expansion of colored Jones polynomials for certain cable of the torus knots.
  Then we show that the above limit is zero based on this estimation.  The whole paper is organized as follows. The next two sections contain some definitions and a formula we shall use---without proof. Section 4 states the main theorem and its corollary. In section 5 and 6 we carry out evaluation of the polynomials and some estimations. The proof of the main theorem is finished in section 7.

\section{Basic definitions and notations}\label{def}

In this section we list some basic definitions and notations. For
details, see Lickorish's book\cite{LBook} and \cite{BHMV}. For a
survey of the skein theory and some generalization,
see\cite{LSurvey}.

Given a compact oriented 3-manifold $M$, and a ring of coefficients
$R=Z[A,A^{-1}]$, define a $R-module$ by
\[SS(M)=the\ free\ module\ generated\ by\ framed\ links\ in\ M\]
where two framed links isotopic to each other are regarded as the same one.

\begin{definition} The skein module associated to $M$ is
\[S(M)= SS(M)/the\ Kauffman\ skein\ relations\]
\end{definition}

The skein relations are the following:

\[\scalebox{.15}{\psfig{file=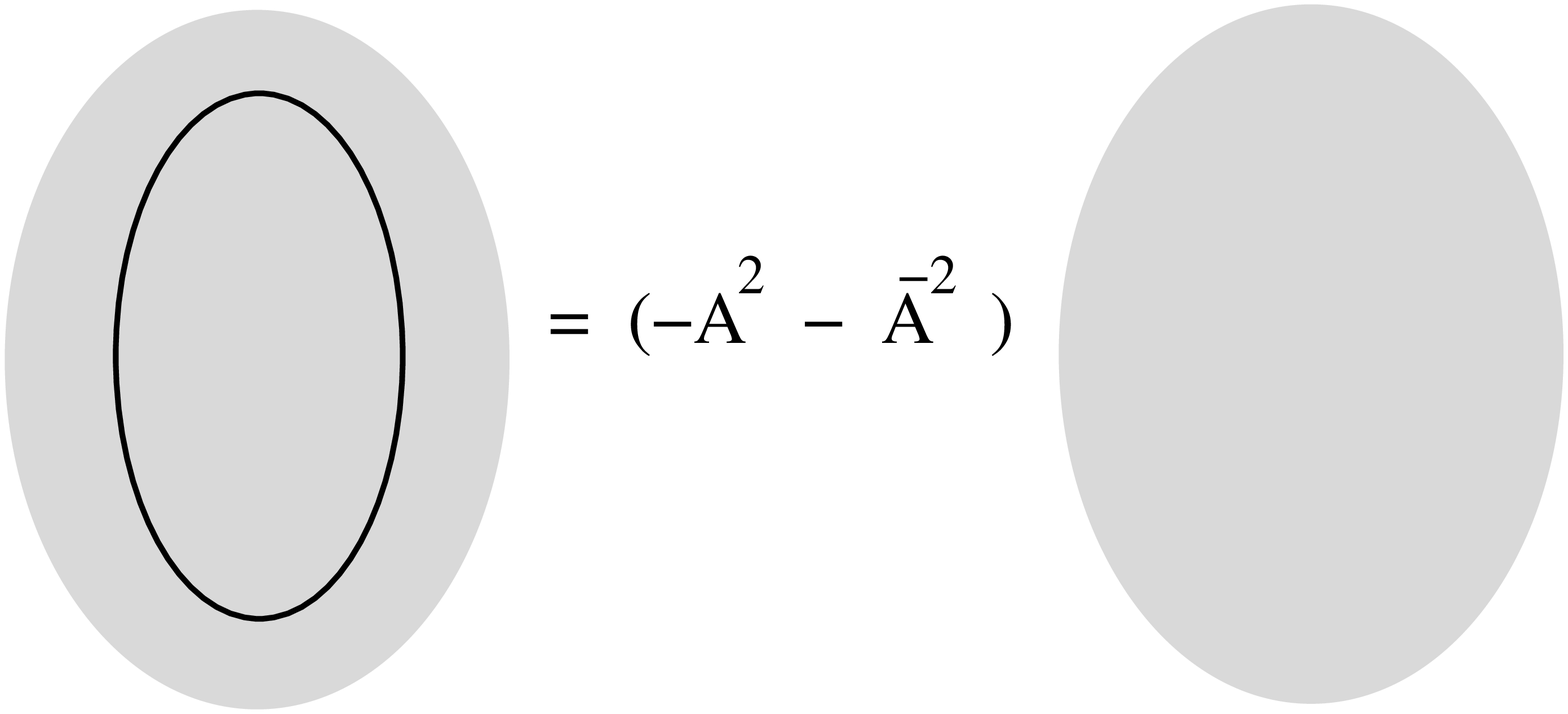}}\]

\[\scalebox{.15}{\psfig{file=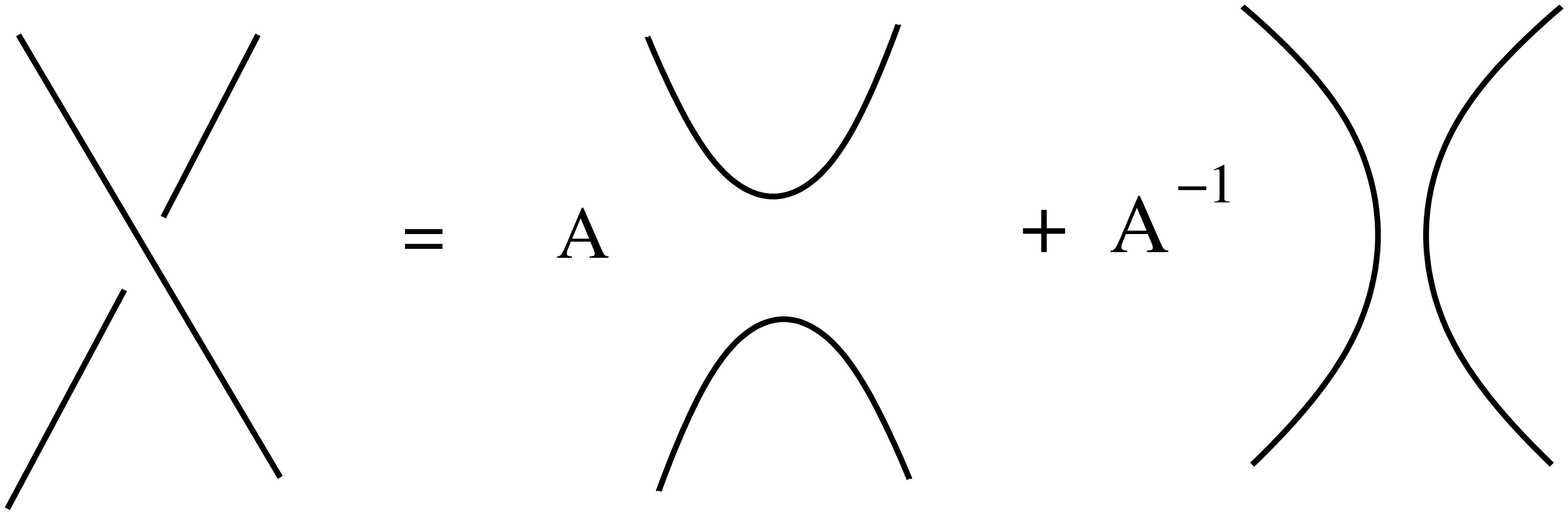}}\]

\subsection{Examples}\label{examples}
We list two examples here.

\subsubsection{The skein module of $S^3$}\label{three sphere}
The skein module of $S^3$ is just equal to $Z[A,A^{-1}]$.--- Given a planar diagram of a famed link $L$ in $S^3$, by applying the skein relation repeatedly to resolve the crossings and remove the unknots, $L$ can be reduced to the empty link multiplied by a Laurent polynomial of $A$. In this way, we associate a polynomial to each framed link in $S^3$. Multiplied by a normalization factor, we will get a polynomial which is independent of the framing of $L$. Up to change of variables, this is just the Jones polynomial of $L$.

\subsubsection{The skein module of $D\times S^1$}\label{solid torus}
The situation in the solid torus is a bit different. Using the skein
relations, we can not reduce every link to the empty one. Certain
essential links remain. And $S(D\times S^1)$ is a free  $R-module$
generated by those elements represented by parallel strings along
the core of the solid torus, see the pictures below.
\[\scalebox{.35}{\psfig{file=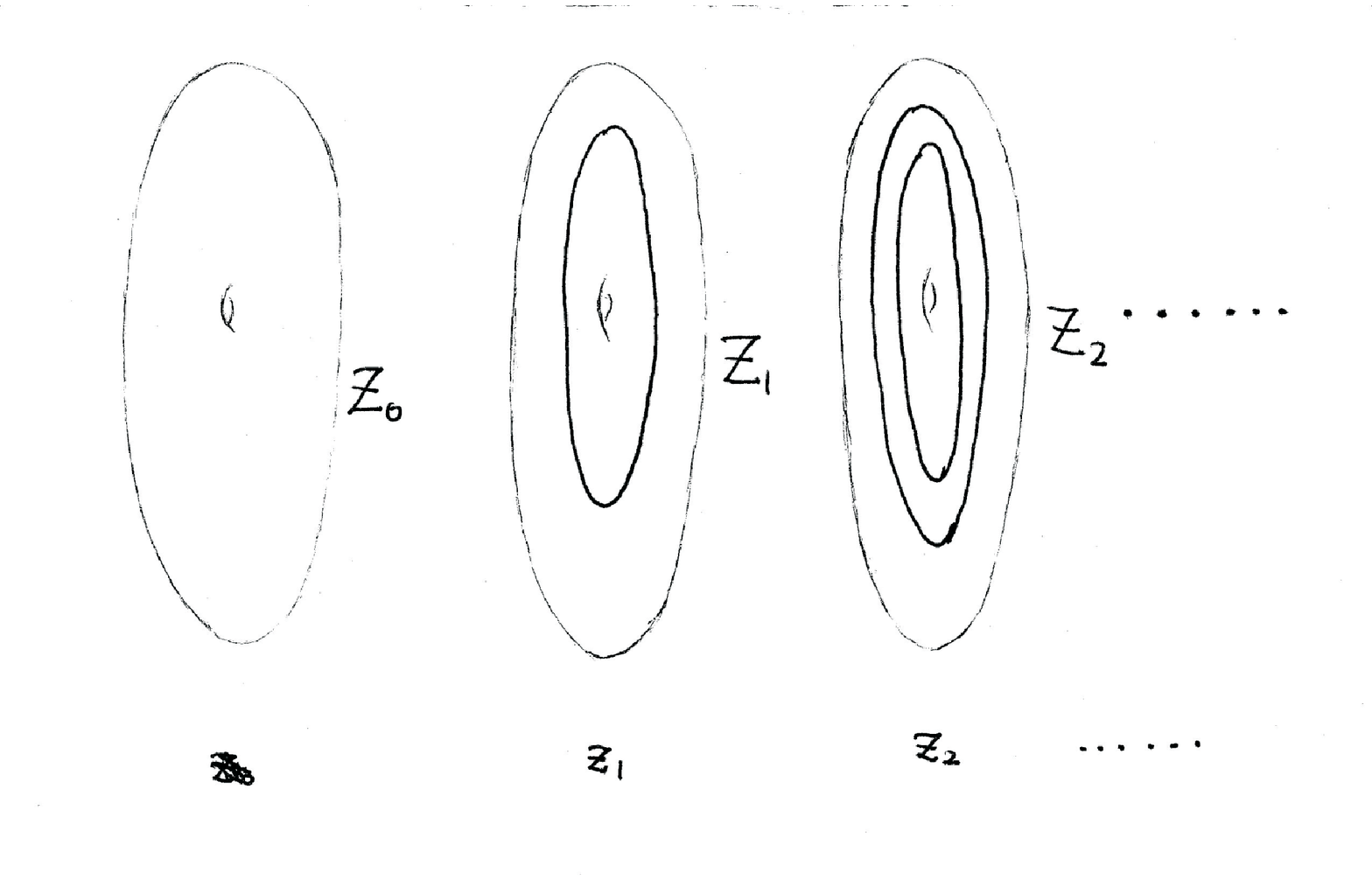}}\]
\subsubsection{The algebra structure of $S(D\times S^1)$ and an alternative basis}\label{algebra and new basis}
The skein module of the standard solid torus $D\times S^1$ has an
algebra structure. Let's define the multiplication now. Fix a curve
on $\partial(D\times S^1)$, say $\{1\}\times S^1$ --- here we
identify $D$ with the unit disk in the complex plane, call it $C_+$.
Pick another curve on $\partial(D\times S^1)$ parallel to $C_+$, say
$\{-1\}\times S^1$, call it  $C_-$. Pick two copies of $D\times
S^1$, glue the first one to the second one by gluing a neighborhood
of $C_+$ in the boundary from the first copy to a neighborhood of
$C_-$ in the boundary from the second copy, we have a new solid
torus. The multiplication of two elements from $S(D\times S^1)$ is
obtained by juxtaposition of two skeins. See the picture below.
\[\scalebox{.45}{\psfig{file=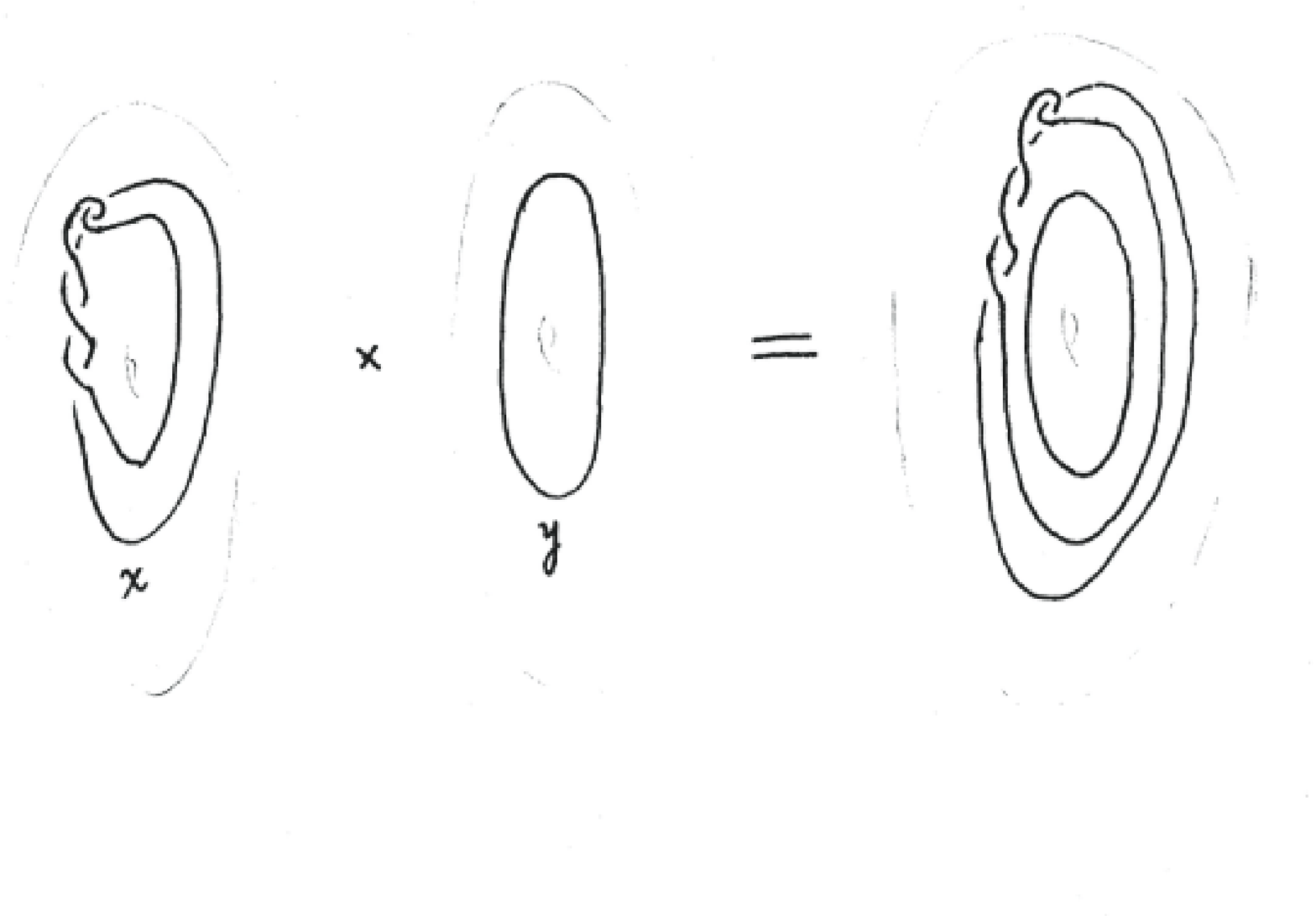}} \]
With this
multiplication, $S(D\times S^1)$ become a commutative $R-algebra$.
It is isomorphic to $R[z_1]$, where $z_1$ is one of the basis
elements described above. We also have $z_n=z_1^n$.

Then we define the following elements $\{e_n|n\ integers\}$ of $S(D\times S^1)$ by $e_0=1,e_1=z_1$ and $z_1e_n=e_{n-1}+e_{n+1}$.
These relations determine $e_n$ for $n\ge 0$ as well as those $e_n$ where $n<0$. And the set $\{e_n|n\ge 0 \}$ is a basis of
$S(D\times S^1)$, since each $e_n$ is a monic polynomial of $z_1$ with degree n. We also have the relation $e_n=-e_{-n-2}$. ---This can be checked by induction.

Next let's describe the multiplication using this new basis. First let's introduce a new definition.
\begin{definition} A triple of nonnegative integers $(l,m,n)$ is called admissible if $l+m+n$ is even and $(l,m,n)$ satisfies the triangle inequality, that is \[ |m-n|\le l \le m+n\]
\end{definition}

Then for $m,n$ nonnegative, we have
\[ e_m e_n=\sum_{(l,m,n)\ admissible} e_l\]
This can also be checked by induction on n.

\subsection{The multilinear bracket}\label{multilinear bracket}
Given an framed link $L$ in $M$ with components $K_1,K_2,...K_m$,
say $K_j$ with framing $\sigma_j$, we can construct multilinear
bracket $\langle *,*,...,*\rangle_L$ in $S(M)$ as following:

the framing $\sigma_j$ of $K_j$ determines a diffeomorphism
\[\tau_j:D\times S^1 \to a\ tubular\ neighbourhood\ of\ K_j\] view
$\tau_j$ as a map $D\times S^1 \to M$, it induces a homomorphism
\[(\tau_j)_*:S(D\times S^1) \to S(M)\]
these $(\tau_j)_*$ combine
together to give us a multi-linear map \[\tau_*: S(D\times
S^1)\times S(D\times S^1)\times ...\times S(D\times S^1)\to S(M)\]
where the product contains $m$ factors. Define $\langle
x_1,x_2,...,x_m\rangle_L$ to be $\tau_*(x_1,x_2,...,x_m)$.

If each component $L$ is colored by some integer, say $K_j$ is colored by $n_j$, then the colored framed link $L$ give us an element in $S(M)$, that is $\langle e_{n_1},e_{n_2},...,e_{n_m}\rangle_L$.

\begin{remark}
If we change the framing of certain components of $L$, while keep the colors unchanged, say $\sigma_1$ is changed to $\sigma_1+b$, then
\[\langle e_{n_1},e_{n_2},...,e_{n_m}\rangle_{L'}=(-1)^{bn_1}A^{b(n_1^2+2n_1)}\langle e_{n_1},e_{n_2},...,e_{n_m}\rangle_L\]
where $L'$ is the same link as $L$ with changed framings.
\end{remark}

\subsubsection{Examples}\label{examples}
If $K$ is a knot in $S^3$ with framing $\sigma$, the n-th colored Jones polynomial of K --$J_n^{\sigma}(A,K)$-- is defined to be $\langle e_n
\rangle_K$,which is an element of $S(S^3)$. Since $S(S^3)$ is just equal to $Z[A,A^{-1}]$, this element is actually a Laurent polynomial of A.  In particular,
\[\langle e_n\rangle_U=(-1)^n[n+1]\]
where $U$ is the unknot with framing zero, and $[m]=\frac{A^{2m}-A^{-2m}}{A^2-A^{-2}}$.

\subsubsection{The framing convention}
Since the linking number of two knots in oriented $S^3$ is defined, the framing of a knot in $S^3$ can be specified by an integer--- the linking number. So for a knot $K$ in $S^3$, we can also say $K$ has framing n, where n is some integer. Among all the framings, zero framing is a special one. If $K$ has framing 0, we will write $J_n(A,K)$ for the colored Jones polynomial instead of $J_n^0(A,K)$.

For an arbitrary $M$, the linking number is not defined, thus the
framing of a knot $K$ in $M$ can not be specified by an integer.
However, under circumstances it is still possible.  For example, $M$
is the solid torus $D\times S^1$ with a chosen curve $C_+$ ---
$\{1\}\times S^1$--- on its boundary. Embed it into $S^3$, then a
framed knot $K$ in $M$  is mapped to a framed knot in $S^3$, and its
framing corresponds to an integer. In general the corresponding
integer will change if we change the embedding. However if we embed
$D\times S^1$ into  $S^3$ in such a way that $C_+$ has linking
number zero with the core of the solid torus after the embedding,
then the corresponding integer remains the same. So we will say a
knot $K$ in $D\times S^1$ has framing n for some integer n in the
sense that the knot is viewed as a knot in $S^3$ through such kind
of embedding of  $D\times S^1$ into $S^3$.

\section{A cable expansion formula}\label{expansion}
\label{expansion}
 In this section, we will follow the framing convention stated previously.

 Let $M$ be the standard solid torus $D \times S^1$, our link $L$ consists of two components:
 $K_1$ is given by \[( \rho \exp(iq\theta), \exp(ip\theta))|_{0\le \theta \le 2\pi}\]
 where $p,q$ are coprime integers, $0<\rho<1$;
 $K_2$ is just the core of the solid torus
 \[(0, \exp(i\theta))|_{0\le \theta \le 2\pi}\]
 Here we identify $D$ with the unit disk in the complex plane and $S^1$ with the unit circle.

 Framings of both components can be described using integers. Fix the framing of $K_2$ to be 0.  The framing $\sigma$ of $K_1$ corresponds to some integer, we also denote it by $\sigma$.  Color $K_2$ by a nonnegative integer $s$, color $K_1$ by a nonnegative integer $N$. Then $\langle e_N,e_s\rangle_L$ is some element of $S(D \times S^1)$, denote it by $T_\sigma^N(p,q;s)$.

 As an element of $S(D \times S^1)$, $T_\sigma^N(p,q;s)$ can be expressed as a linear combination of the basis elements $\{e_l\}|_{l\ge 0}$. The coefficients before $e_l$ is denoted by
 $g_\sigma^l(p,q;s)$, i.e. we have
 \[T_\sigma^N(p,q;s)=\sum_{l\ge 0} g_\sigma^l(p,q;s)e_l\]

 Note we didn't restrict the range of $l$ except that $l\ge 0$. Actually these coefficients
 are all 0 but finitely many of them.
 Further they are all members of the ring of Laurent
 polynomials $Z[A,A^{-1}]$. The exact formula for these coefficients are the following:

 \begin{theorem} Assume $p>0$,let $\epsilon, \beta, \alpha$ be such that $\epsilon=[\frac{q}{p}]$, $q=\epsilon p+\beta$ with
$0\le \beta<p$, $\alpha=p\cdot\beta$.Then for framing $\sigma=pq$, we have
\begin{align} g_\sigma^l(p,q;s)&=(-1)^{qN} A^{\epsilon
(-s^2-2s+l^2+2l)}\notag\\
&\times\sum_{ \substack {k=-N\\ k+N\ even}}^N A^{\alpha {k}^2+2\beta k
(s+1)}\{\delta(l-s-pk)-\delta(l+s+2+pk)\}\notag\end{align}
\end{theorem}

It seems to be complicated. But note there are at most two nonzero terms in this summation. Since for each value of $k$, the
summand is nonzero only if $l-s-pk=0$ or $l+s+2+pk=0$.

Substitute the expression for $g$'s to the expansion of $T_\sigma^N(p,q;s)$, we have
\begin{theorem}For framing $\sigma=pq$,
\[T_\sigma^N(p,q;s)=(-1)^{qN}\sum_{ \substack {k=-N\\ k+N\ even}}^N A^{pqk^2+2qk(s+1)} e_{pk+s}\]
\end{theorem}

The case $s=0$ is an old result back to 90's. For cabling such that $p,q$ not necessarily coprime, a formula is available in \cite{VV}. In the author's thesis\cite{QL} a proof of the above formulas is given.

\section{Main theorem}

From the result mentioned in the previous section , we have
\[T_\sigma^N (p,q;0)=(-1)^{qN}\sum_{ \substack {k=-N\\ k+N\ even}}^N  A^{pqk^2+2qk} e_{pk}\]
where N is the color of the $(p,q)$-cable ($N\ge 0$), and the framing $\sigma$ of the cable is $pq$.

Let $K'$ be the $(p,q)$-cabling of an arbitrary knot $K$. Then
\[ J_N^{\sigma}(A,K')=(-1)^{qN}\sum_{ \substack {k=-N\\ k+N\ even}}^N  A^{pqk^2+2qk} J_{pk}(A,K)\tag {*}\] Here
the formula is for $N\ge 0$, when $N$ is negative, we shall use a symmetric property of $e_N$: \[e_N=-e_{-N-2}\]
So that
\[J_N(A,K')=-J_{-N-2}(A,K')\]

Let $T$ be the $(p_2,q_2)$ torus knots, $K$ is the $(p_1,q_1)$-cabling of $T$, then
\begin{align}
J_N^{\sigma} (A,K)
    &=(-1)^{q_1N}\sum_{ \substack {k=-N\\ k+N\ even}}^N A^{p_1q_1{k}^2+2q_1k} J_{p_1k}(A,T)\notag\\
    &=(-1)^{q_1N}\sum_{ \substack {k=-N\\ k+N\ even}}^N A^{p_1q_1{k}^2+2q_1k} \left[(-1)^{p_1k}A^{p_1k(p_1k+2)}\right]^{-\sigma'}
    J_{p_1k}^{\sigma'}(A,T)\notag\\
    &\qquad where\ \sigma'=p_2q_2\notag
    \end{align}

Applying the cable expansion formula (*) once again, we have:

\begin{proposition}
\begin{align}
J_N^\sigma (A,K)=&(-1)^{\alpha N}\sum_{ \substack{-N \le k\le N\\k+N\ even}} A^{\beta p_1k(p_1k+2)}sign(p_1k+1)\notag\\
&\cdot\sum_{ \substack{1-|p_1k+1| \le k'\le |p_1k+1|-1\\k'+p_1k\ even}}A^{\gamma p_2k'(p_2k'+2)}[p_2k'+1]\notag
\end{align}
    where\ $\alpha=q_1-p_1p_2q_2 +q_2p_1+p_2p_1$, $\beta =\frac{q_1}{p_1}-p_2q_2$, $\gamma=\frac{q_2}{p_2}$, $[m]=\frac{A^{2m}-A^{-2m}}{A^2-A^{-2}}$.
\end{proposition}

We need to evaluate $\frac{J_N^\sigma (A,K)}{[N+1]}$ at $A=\exp\left(\frac{\pi i}{2(N+1)}\right)$, and then consider the
asymptotic behavior when $N\to\infty$. The main result is the following:

\begin{theorem} Under the condition that $\beta \gamma >0$, there is a sequence $\alpha_1 < \alpha_2 <...$ and $\alpha_n
\to \infty$, such that for any integer $m>0$,
\begin{align}
d_N\cdot\left(\frac{J_N^\sigma (A,K)}{[N+1]}\right)_{A=\exp\left(\frac{\pi i}{2(N+1)}\right)}&=\sum_{n=1}^m \frac{C_{n,N}}{(N+1)^{\alpha _n}}+o(\frac1{(N+1)^{\alpha_m}})\notag\\
&\qquad\qquad\qquad N\to \infty \notag\end{align}  The coefficient $c_{n,N}$ depend on $N$ periodically, and all of them have a common
period. And $d_N\to 1$. Furthermore, at least one of the coefficients is nonzero.
\end{theorem}

\begin{corollary} If $\beta \gamma >0$, then \[\lim_{N\to \infty} \frac{\log\left|\left(\frac{J_N^\sigma
(A,K)}{[N+1]}\right)_{A=\exp\left(\frac{\pi i}{2(N+1)}\right)}\right|}N=0\] That is, volume conjecture is true for certain
cable of the torus knots.
\end{corollary}

\begin{proof}Denote \[\left(\frac{J_N^\sigma (A,K)}{[N+1]}\right)_{A=\exp\left(\frac{\pi i}{2(N+1)}\right)}\] by $\kappa_N$. Since those $C$'s have a common period, say
$T$, by decomposing the sequence \{$\kappa_N$\} into $T$ subsequences, we can assume that those $C$'s are independent of $N$.
Pick the first nonzero $C$, say $C_{m_0}$. Then
\[d_N\cdot\kappa_N=\frac{C_{m_0}}{(N+1)^{\alpha_{m_0}}}+o(\frac1{(N+1)^{\alpha_{m_0}}})\]
that is \[\kappa_N (N+1)^{\alpha_{m_0}} \to C_{m_0}\] so \[\log|\kappa_N|+ \alpha_{m_0}\log(N+1)\to \log|C_{m_0}|\] and our
claim follows from it.
\end{proof}

In the next two sections we will prepare for the proof of the main theorem. And the proof is finished in section 7. The analytical treatment here follows that of Kashaev \cite{KT}. 

\section{Integral representation of the polynomial and its value at root of unity}

\subsection{Integral representation of the double sum}
Now let's first evaluate $\frac{J_N^\sigma (A,K)}{[N+1]}$ at $A=\exp\frac{h}2$, where $h$ is some complex number. From proposition 1 we have

 \begin{align} (-1)^{\alpha N}&\exp\left(\frac{\beta h+\gamma h}2\right)\sinh((N+1)h)\left(\frac{J_N^\sigma
(A,K)}{[N+1]}\right)_{A=\exp\frac{h}2}\tag{5.1}\\
&=\sum_{k}\sum_{k'}\exp\left(\frac{\beta(p_1k+1)^2 h}2\right)\exp\left(\frac{\gamma(p_2k'+1)^2
h}2\right)\notag\\
&\qquad\qquad\cdot sign(p_1k+1)\sinh((p_2k'+1)h)\notag\end{align}

Now we use an integral expression of $\exp(h'w^2)$,
\[\exp(h'w^2)=\frac 1{\sqrt{\pi h'}}\int_{C_\phi} \exp\left(-\frac{z^2}{h'}-2zw\right)\ dz\]
where $C_\phi$ is a line in the $z$-plane, $C_\phi=\lbrace z|argz=\phi\ or\ \phi+\pi \rbrace$, and $\phi$ satisfies
\[-\frac {\pi}2 +arg(h')< 2\phi < \frac {\pi}2 +arg(h')\]

\begin{proposition} Denote the right hand side of (5.1) by $S_N(h)$, we have

\begin{align} S_N(h)&=\frac 2{\pi h\sqrt{\beta\gamma}}\int_{C_{\phi_1}}\int_{C_{\phi_2}}\,dz_1 \exp\left(-\frac{2z_1^2}{\beta h}\right)\,dz_2
\exp\left(-\frac{2z_2^2}{\gamma h}\right)\notag\\
&\qquad\cdot\sum_{k}\sum_{k'}\exp\left(-2z_1(p_1k+1)\right)\exp\left(-2z_2(p_2k'+1)\right)\notag\\
&\qquad\cdot sign(p_1k+1)\sinh((p_2k'+1)h)\notag\end{align}
where $C_{\phi_j}$ is a line in the $z_j$-plane.
\end{proposition}

From now on, we shall assume $\beta\gamma >0$, and choose $\phi_1=\phi_2$. The desired asymptotic expansion will be obtained under this condition.

To evaluate a sum of the following form:

\[\int_{C_{\phi}}\,dz \sum_{ \substack{-m \le k\le m\\k+m\ even}} \exp\left(-\frac{z^2}{h'}\right)
\exp\left(-2z(pk+1)\right)\sinh(\delta(pk+1))\]

we calculate the sum over $k$ first,--- this is essentially a geometric series, we get:

\[\int_{C_{\phi}}\,dz \exp\left(-\frac{z^2}{h'}\right)\sinh(-2z+\delta)\frac {\sinh(p(m+1)(2z-\delta))}{\sinh(p(2z-\delta))}\]

Then we shift the path of integral to $C_{\phi}+\delta/2$, and apply change of variables $z\mapsto z+\delta/2$, we get:

\[\int_{C_{\phi}}\,dz \exp\left(-\frac{(z+\frac {\delta}2)^2}{h'}\right)\sinh(-2z)\frac {\sinh(2p(m+1)z)}{\sinh(2pz))}\]

Next we expand $(z+\frac {\delta}2)^2$, move the constant term out of the integral, and move down the linear term of $z$ from the exponential,--- this step is possible since \[\sinh(-2z)\frac {\sinh(2p(m+1)z)}{\sinh(2pz))}\] is an odd function of $z$, finally we have:

\[\exp\left(-\frac{\delta ^2}{4h'}\right)\int_{C_{\phi}}\,dz
\exp\left(-\frac{z^2}{h'}\right)\frac{\sinh(\frac{z\delta}{h'})\sinh(2z)}{\sinh(p(2z))}\sinh(p(m+1)(2z))\]

Returning to $S_N$, follow the above process of evaluation twice, we have the following integral representation of $S_N$:

\begin{proposition}

\begin{align}
S_N(h)&=\frac 2{\pi h\sqrt{\beta\gamma}}\frac 1{p_2} \exp\left(-\frac
h{2\gamma}\right)\int_{C_{\phi_1}}\int_{C_{\phi_2}}\,dz_1\,dz_2\psi_1(z_1)\psi_2(z_2)\notag\\
&\qquad\qquad\qquad\cdot \exp\left(-\frac{2z_2^2}{p_2^2\gamma h}-\frac{2(z_1-z_2)^2}{\beta h}+2p_1(N+1)z_1\right)\notag\end{align}

where \[\psi_1(z_1)=\frac{\sinh(2z_1)}{\sinh(2p_1z_1)}\]
\[\psi_2(z_2)=\frac{\sinh(\frac{2z_2}{p_2\gamma})\sinh(\frac {2z_2}{p_2})}{\sinh(2z_2)}\]
\end{proposition}

Note the double integral is absolute convergent since we assume $\beta\gamma >0$ and choose $\phi_1=\phi_2$.

\subsection{Evaluation at root of unity}
From (5.1), we see that $S_N\left(\frac{\pi i}{N+1}\right)=0$, since the left hand side of (5.1) contains a factor $\sinh((N+1)h)$.
If we differentiate (5.1) w.r.t. $h$ and then set $h=\frac{\pi i}{N+1}$, we have
\begin{align}
&\qquad (-1)^{\alpha N}\exp\left(\frac {\beta+\gamma}{2(N+1)}\pi i\right)\cdot(-(N+1))\left(\frac{J_N^\sigma
(A,K)}{[N+1]}\right)_{A=\exp\left(\frac{\pi i}{2(N+1)}\right)}\tag{5.2}\\
&=\left(\frac d{dh} S_N(h)\right)_{h=\frac{\pi i}{N+1}}\notag\end{align}

So for our purpose, we need to consider $\left(\frac d{dh} S_N(h)\right)_{h=\frac{\pi i}{N+1}}$.
Use the integral representation from proposition 3, we have:

\begin{proposition}
\begin{align}
&\qquad\left(\frac d{dh} S_N(h)\right)_{h=\frac{\pi i}{N+1}}\notag\\
&=\frac{2(N+1)^3}{p_2\sqrt{\beta\gamma}\pi ^4 i^3}\exp\left(-\frac {\pi i}{2\gamma
(N+1)}\right)\int_{C_{\phi_1}}\int_{C_{\phi_2}}\,dz_1\,dz_2\psi_1(z_1)\psi_2(z_2)\notag\\
&\qquad\qquad\qquad\cdot \exp[(N+1)\theta (z_1,z_2)]\cdot \left(\frac{2z_2^2}{p_2^2\gamma
}+\frac{2(z_1-z_2)^2}{\beta}\right)\notag\end{align}

where \[\theta (z_1,z_2)=-\frac{2z_2^2}{p_2^2\gamma \pi i}-\frac{2(z_1-z_2)^2}{\beta \pi i}+2p_1z_1\]
\[\psi_1(z_1)=\frac{\sinh(2z_1)}{\sinh(2p_1z_1)}\]
\[\psi_2(z_2)=\frac{\sinh(\frac{2z_2}{p_2\gamma})\sinh(\frac {2z_2}{p_2})}{\sinh(2z_2)}\]
\end{proposition}

Now we have chosen a value of $h$, consequently we shall choose the value of $\phi_1$ and $\phi_2$.
In case $\beta>0$ ---hence $\gamma>0$, we choose:

\[\phi_1=\phi_2=\pi/4\]

or any value strictly between $0$ and $\frac{\pi}2$. In case $\beta<0$, we choose:

\[\phi_1=\phi_2=-\pi/4\]

or any value strictly between $0$ and $-\frac{\pi}2$.

\section{Study of the asymptotic behavior of the double integral}\label{asymptotics}

 Denote the double integral in proposition 4 by $I_N$. We need to show that for any integer $M>0$, $I_N$ has an expansion

\[I_N=\sum_{m=0}^M \frac {C_{m,N}}{(N+1)^{m/2}}+o\left(\frac 1{(N+1)^{M/2}}\right)\tag{A}\]

where the coefficients $C_{m,N}$ depend on $N$, but is a periodic function of N. All of them have a common period $T$ which
depends on $p_1,q_1,p_2,q_2$ only. Furthermore at least one of the coefficients is nonzero.

$\theta(z_1,z_2)$ has a unique critical point
\[(w_1,w_2)=\left(\frac {p_1p_2^2\gamma+p_1\beta}2 \pi i,\ \frac{p_1p_2^2\gamma}2 \pi i\right)=\left(\frac{q_1\pi i}2,\ \frac{p_1p_2q_2\pi i}2 \right)\]

In the double integral $I_N$, we shall shift $C_{\phi_j}$ to $C_{\phi_j}+w_j$ ($j=1,2$). So we must consider the possible poles
of $\psi_j(z_j)$ in the strip region bounded by $C_{\phi_j}$ and $C_{\phi_j}+w_j$.

$\psi_1(z_1)$ has possible poles $\xi_m=\frac m{2p_1}\pi i$ within $C_{\phi_1}$ and $C_{\phi_1}+w_1$, here $m$ is an integer, contained in the open interval from 0 to $p_1q_1$.

$\psi_2(z_2)$ has possible poles $\eta_n=\frac n{2}\pi i$ within $C_{\phi_2}$ and $C_{\phi_2}+w_2$, here $n$ is an integer, contained in the open interval from 0 to $p_1p_2q_2$.

Let \[F_N(z_1,z_2)=\left(\frac{2z_2^2}{p_2^2\gamma
}+\frac{2(z_1-z_2)^2}{\beta}\right)\exp[(N+1)\theta (z_1,z_2)]\]

Then \begin{proposition}
\begin{align}
I_N&=\int_{C_{\phi_1}}\int_{C_{\phi_2}}\,dz_1\,dz_2 \psi_1(z_1)\psi_2(z_2)F_N(z_1,z_2)\notag\\
&=\int_{C_{\phi_1+w_1}}\int_{C_{\phi_2+w_2}}\,dz_1\,dz_2 \psi_1(z_1)\psi_2(z_2)F_N(z_1,z_2)\notag\\
& +\sum_m (2\pi i) Res(\psi_1,\xi_m)\int_{C_{\phi_2+w_2}}\,dz_2\psi_2(z_2)F_N(\xi_m,z_2)\notag\\
&+\sum_n (2\pi i) Res(\psi_2,\eta_n)\int_{C_{\phi_1+w_1}}\,dz_1\psi_1(z_1)F_N(z_1,\eta_n)\notag\\
& +\sum_{m,n}(2\pi i)^2 Res(\psi_1,\xi_m)Res(\psi_2,\eta_n)F_N(\xi_m,\eta_n)\notag\end{align}
\end{proposition}

Next we will consider the asymptotic expansion of each of the four terms and show that they don't cancel each other.

\subsection{The double summation}\label{the double summation}

 First consider the double summation. One can check that $\theta (\xi_m,\eta_n)=\frac s{t} \pi i$, where $s,t$ are integers,
depending on $p_1,q_1,p_2,q_2$ and $k,l$.\quad So $F_N(\xi_m,\eta_n)$ depends on N periodically, and there is a common period of them, which is independent of $N$.

So the double summation is a quantity depend on N periodically.

\subsection{The summations over n}\label{the two summations}

Now consider the integral
\[\int_{C_{\phi_1+w_1}}\,dz_1\psi_1(z_1)F_N(z_1,\eta_n)\]
$\theta (z_1,\eta_n)$ has a unique critical point \[\zeta=\frac{p_1\beta+n}{2}\pi i\] and
\[\theta (z_1,\eta_n)=-\frac 2{\beta\pi i} (z_1-\zeta)^2 + \theta (\zeta,\eta_n)\]
As before, we shift $C_{\phi_1+w_1}$ to $C_{\phi_1+\zeta}$. Note that $\zeta$ is not a pole of $\psi_1(z)$ since $p_1\beta+n$ is an integer.

The difference between the integrals over the two paths is a sum of residues of $\psi_1(z_1)F_N(z_1,\eta_n)$ at poles of $\psi_1(z)$. As before, one can check that these residues depend on $N$ periodically.

As for the integral over $C_{\phi_1+\zeta}$, we have:

\begin{align}
&\qquad \int_{C_{\phi_1+\zeta}}\,dz_1\psi_1(z_1)F_N(z_1,\eta_n)\notag\\
&=\exp[(N+1)\theta(\zeta,\eta_l)]\int_{C_{\phi_1}}\,dz_1\Psi(z_1+\zeta)\exp\left[-\frac {2(N+1)}{\beta\pi i}
z_1^2\right]\notag\end{align}

where \[\Psi(w)=\psi_1(w)\left(\frac{2\eta_l^2}{p_2^2\gamma
}+\frac{2(w-\eta_l)^2}{\beta}\right)\] is holomorphic along the line $C_{\phi_1}+\zeta$.

The asymptotic behavior of the last integral is given by the following lemma:

\begin{lemma} Let $\alpha$ be a complex number, $Re\alpha>0$, $g(x)$ is a smooth function, and $|g(x)|<Ce^{\lambda |x|}$
for some constant $C, \lambda$, then for any integer $M>0$,
\begin{align}
&\qquad \int_{-\infty}^{\infty} \exp[-(N+1)\alpha x^2] g(x)\, dx\notag\\
&=\sum_{m=0}^M \frac{g^{(2m)}(0)}{(2m)!}\cdot \frac 1{(N+1)^{m+1/2}}\int_{-\infty}^{\infty} \exp(-\alpha x^2) x^{2m}\,
dx+o\left(\frac 1{(N+1)^{M+1/2}}\right)\notag\\
&\qquad\qquad\qquad\qquad\qquad\qquad\qquad\qquad\qquad N\to \infty \notag\end{align}
\end{lemma}

Put these pieces together, we have
\begin{proposition}For any integer $M>0$,
\begin{align}
&\qquad\int_{C_{\phi_1+w_1}}\,dz_1\psi_1(z_1)\left(\frac{2\eta_l^2}{p_2^2\gamma
}+\frac{2(z_1-\eta_l)^2}{\beta}\right)\exp[(N+1)\theta (z_1,\eta_l)]\notag\\
&=\sum_{m=0}^{M} \frac{C_{m,N}^\prime}{(N+1)^{m+1/2}}+C_N^\prime+o\left(\frac 1{(N+1)^{M+1/2}}\right)\notag\end{align}
where $C_{m,N}^\prime$,$C_N^\prime$ depend on N periodically and have a common period. Also note that in the expansion the
exponents are all {\it half integers\/}.
\end{proposition}

\subsection{The summation over m}
The integral \[\int_{C_{\phi_2+w_2}}\,dz_2\psi_2(z_2)F_N(\xi_m, z_2)\] is similar. There is a slight difference though. The critical point of $\theta (\xi_m, z_2)$ is  \[\zeta'=\frac{p_2q_2m}{2q_1}\pi i\]
Unlike the previous case, $\zeta'$ could be a pole of $\psi_2(z_2)$.  Thus we need to subtract the singular part of $\psi_2(z_2)$ in order to shift the path of integral to $C_{\phi_2+\zeta'}$. Note that
$\psi_2(z_2)$ has only simple poles. So in addition to the estimation procedure listed above, we also need to consider  integral of the following form
\[\int_{C_{\phi_2+w_2}}\,dz_2\frac1{z_2-\zeta'}F_N(\xi_m, z_2)\]

We need the following result
\begin{lemma}
Let $\alpha$ be a complex number, $Re\alpha>0$, $\epsilon>0$, then
\begin{align}
&\quad \int_{R+i\epsilon}\, dz\frac1{z} \exp[-(N+1)\alpha z^2]\notag\\
&= -\int_{R-i\epsilon}\, dz\frac1{z} \exp[-(N+1)\alpha z^2]\notag\\
&=-\pi i\notag
\end{align}
where $R$ stands for the real line.
\end{lemma}

Use this result, we can check that
\[\int_{C_{\phi_2+w_2}}\,dz_2\frac1{z_2-\zeta'}F_N(\xi_m, z_2)\]
contains two parts, one depends on $N$ periodically, the other one has the order $(N+1)^{-\frac1{2}}$.

So for
\[\int_{C_{\phi_2+w_2}}\,dz_2\psi_2(z_2)F_N(\xi_m, z_2)\]
we also get an asymptotic expansion of the same form as proposition 6.

\subsection{The double integral}\label{the double integral}

For the first part of $I_N$, we denote it by $I_N^\prime$.

By change of variables $z_1\mapsto z_1+w_1$, $z_2\mapsto z_2+w_2$ and some elementary calculations, we have:

\begin{align}
I_N^\prime&=\exp[(N+1)\theta(w_1,w_2)]\int_{C_{\phi_1}}\int_{C_{\phi_2}}\,dz_1\,dz_2\psi_1(z_1)\psi_2(z_2)(2p_1\pi i)z_1\notag\\
&\qquad\qquad\qquad\cdot \exp\left[(N+1)\left(-\frac {2(z_1-z_2)^2}{\beta \pi i} -\frac {2z_2^2}{p_2^2\gamma\pi
i}\right)\right]\notag\end{align}

The asymptotic behavior of the above double integral is given by:

\begin{lemma} Let $\alpha$ be a complex number, $Re\alpha>0$, $f(x_1,x_2)$ is a positive definite quadratic form of
$x_1,x_2$, $g(x_1,x_2)$ is a smooth function, and $|g(x_1,x_2)|< Ce^{\lambda_1 |x_1|+ \lambda_2 |x_2|}$ for some constants
$C,\lambda_1, \lambda_2$, then for any integer $M>0$, we have
\begin{align}
&\qquad\int_{-\infty}^{\infty}\int_{-\infty}^{\infty}\exp[-(N+1)\alpha f(x_1,x_2)]g(x_1,x_2)\,dx_1\,dx_2\notag\\
&=\sum_{m=0}^M \frac 1{(N+1)^{m+1}}\notag\\
&\quad\cdot\left(\sum_{k=0}^m \frac 1{k!(2m-k)!}\frac {\partial ^{2m}g}{\partial x_1^k \partial
x_2^{2m-k}}(0,0)\int_{-\infty}^{\infty}\int_{-\infty}^{\infty}\exp[-\alpha f(x_1,x_2)]x_1^k
x_2^{2m-k}\,dx_1\,dx_2\right)\notag\\
&\qquad\qquad\qquad +o\left(\frac 1{(N+1)^{M+1}}\right)\notag\\
&\qquad\qquad\qquad\qquad\qquad\qquad\qquad\qquad\qquad N\to \infty\notag\end{align}
\end{lemma}

Put these together, \begin{proposition}For any integer $M>0$,
\[I_N^\prime=\sum_{m=0}^M \frac{C_{m,N}^{''}}{(N+1)^{m+2}}+o\left(\frac 1{(N+1)^{M+2}}\right),\qquad\qquad N\to \infty\]
where $C_{m,N}^{''}$ depend on N periodically and have a common period.
Note that the exponents in this expansion are all {\it integers\/}.
\end{proposition}

Up to now, we have obtained asymptotic expansions of the four parts of $I_N$, hence an asymptotic expansions of $I_N$ itself.

\section{Proof of theorem 3}\label{proof}

\begin{proof}[Proof of Theorem 3] Now back to $I_N$, it decomposes into four parts ---a double integral, two summation of integrals
and a double summation. By expanding each of the four parts, we have
shown that (A) is true. The only question remains is whether there is a non-vanishing term. Note that those terms with positive integer
exponents come from $I_N^\prime$ only, and there is a nonzero term in the expansion of $I_N^\prime$.---Actually the first term in this expansion is nonzero as one can check with the help of the lemma below:

\begin{lemma}Let $\alpha$ be a complex number, $Re\alpha>0$, $f(x_1,x_2)$ is a positive definite quadratic form of
$x_1,x_2$,
\[\int_{-\infty}^{\infty}\int_{-\infty}^{\infty}\exp[-\alpha f(x_1,x_2)]x_1^k x_2^l\,dx_1\,dx_2=\left(\frac{\partial ^k}{\partial J_1^k}\frac{\partial ^l}{\partial J_2^l}E(J_1,J_2)\right)_{(J_1,J_2)=(0,0)}\]
where \[E(J_1,J_2)=\int_{-\infty}^{\infty}\int_{-\infty}^{\infty}\exp[-\alpha f(x_1,x_2)+J_1x_1+J_2x_2]\,dx_1\,dx_2\]
\end{lemma}

Thus by explicit calculation, we can check that

\[\int_{C_{\phi_1}}\int_{C_{\phi_2}}\,dz_1\,dz_2 z_2z_1 \exp\left(-\frac {2(z_1-z_2)^2}{\beta \pi i} -\frac {2z_2^2}{p_2^2\gamma\pi
i}\right)\ne 0\]

  Hence under the condition that $\beta\gamma>0$, combining (5.2), proposition 4 and (A),
finally we have an nontrivial expansion of \[\left(\frac{J_N^\sigma
(A,K)}{[N+1]}\right)_{A=\exp\left(\frac{\pi i}{2(N+1)}\right)}\] of the form claimed in theorem 3.
\end{proof}

\end{document}